\documentclass[11pt]{article} 
\usepackage{amsmath,amssymb, latexsym,graphicx, color, url}
\newtheorem{defin}{}
\newtheorem{saetze}[defin]{}

\newtheorem{questi}[defin]{}
\newtheorem{lemmas}[defin]{}
\newtheorem{folger}[defin]{}
\newtheorem{bemerk}[defin]{}

\newtheorem{conjecture}{Conjecture}

\newenvironment{theorem}  {\begin{saetze}\it {\bf Theorem:}}{\end{saetze}}
\newenvironment{question} {\begin{questi}\it {\bf Question:}}{\end{questi}}
\newenvironment{lemma}    {\begin{lemmas}\it {\bf Lemma:}}{\end{lemmas}}

\newenvironment{remark}   {\begin{bemerk}\rm {\it Remark:}}{\end{bemerk}}
\newenvironment{proof}    {{\it Proof}:}{{\hfill \fillbox \bigskip}}

\newcommand{\fillbox}{\mbox{$\bullet$}}
\newcommand{\nlab}{\overline}
\newcommand{\ra}{\rightarrow}

\newcommand{\N}{\mathbb N}

\newcommand{\G}{\mathcal G}

\newcommand{\T}{\mathcal T}

\renewcommand{\L}{\mathcal L}
\newcommand{\M}{\mathcal M}

\newcommand{\Ann}{\operatorname{Ann}}

\newenvironment{items}{\begin{list}{$\alph{item})$}
{\labelwidth18pt \leftmargin18pt \topsep3pt \itemsep1pt \parsep0pt}}
{\end{list}}

\newcommand{\bulit}{\item[$\bullet$]}

\begin{document}

\title{Coclass theory for nilpotent semigroups \\ 
       via their associated algebras}
\author{Andreas Distler\footnote{The first author is supported by 
  the project PTDC/MAT/101993/2008 of Centro de \'Algebra da Universidade 
  de Lisboa, financed by FCT and FEDER.} $\;$ and Bettina Eick}
\date{\today}
\maketitle

\begin{abstract}
Coclass theory has been a highly successful approach towards the 
investigation and classification of finite nilpotent groups. Here we
suggest a similar approach for finite nilpotent semigroups. This 
differs from the group theory setting in that we additionally 
use certain algebras associated to the considered semigroups. We propose 
a series of conjectures on our suggested approach. If these become 
theorems, then this would reduce the classification of nilpotent semigroups 
of a fixed coclass to a finite calculation. Our conjectures are supported 
by the classification of nilpotent semigroups of coclass 0 and 1. 
Computational experiments suggest that the conjectures also hold for the
nilpotent semigroups of coclass 2 and 3.
\end{abstract}

\section{Introduction}

A semigroup $O$ or an associative algebra $O$ is called {\em nilpotent} if 
there exists an integer $c$ so that every product of $c+1$ elements equals 
zero. The least integer $c$ with this property is the {\em class} $cl(O)$ 
of $O$; equivalently, the class of $O$ is the length of series of powers
\[ O > O^2 > \ldots > O^c > O^{c+1} = \{0\}. \]
The {\em coclass} of a finite nilpotent semigroup $O$ with $n$ non-zero 
elements or a finite dimensional nilpotent algebra $O$ of dimension $n$ 
is defined via $cc(O) = n-cl(O)$.

For a semigroup $S$ and a field $K$ we denote with $K[S]$ the semigroup 
algebra defined by $K$ and $S$. This is an associative algebra of dimension
$|S|$. If $S$ has a zero element $z$, then the subspace $U$ of $K[S]$ 
generated by $z$ is an ideal in $K[S]$. We call $K[S]/U$ the {\em 
contracted semigroup algebra} defined by $K$ and $S$ and denote it by 
$KS$. If $S$ is a finite nilpotent semigroup, then $KS$ is a nilpotent 
algebra of the same class and coclass as $S$.

Our first aim in this note is to suggest a general approach towards a
classification up to isomorphism of nilpotent semigroups of a fixed 
coclass. For this purpose we choose an arbitrary field $K$ and we define
a directed labelled graph $\G_{r, K}$ as follows: the vertices of $\G_{r, K}$ 
correspond one-to-one to the isomorphism types of algebras $KS$ for the 
nilpotent semigroups $S$ of coclass $r$; two vertices $A$ and $B$ are 
adjoined by a directed edge $A \ra B$ if $B/B^c \cong A$, where $c$ is 
the class of $B$; each vertex $A$ in $\G_{r,K}$ is labelled by the 
number of isomorphism types of semigroups $S$ of coclass $r$ 
with $A \cong KS$. Illustrations of parts of such graphs can be found 
as Figure~\ref{figcc1} on page~\pageref{cc1} and as Figure~\ref{figcc2d2}
on~\pageref{cc2}.

We have investigated various of the graphs $\G_{r,K}$ and we observed that
all these graphs share the same general features. We formulate a sequence of 
conjectures and theorems describing these features. If our conjectures become 
theorems, then this would provide the ground for a new approach towards the 
classification and investigation of nilpotent semigroups by coclass. In 
particular, it would show how the classification of the infinitely many 
nilpotent semigroups of a fixed coclass reduces to a finite calculation.

As a second aim in this note we exhibit some graphs $\G_{r,K}$ explicitly 
to illustrate our conjectures. We have determined the graphs $\G_{0,K}$ and 
$\G_{1,K}$ for all fields $K$ using the classification of the nilpotent 
semigroups of small coclass in \cite{Dis10, Dis11}; see Sections 
\ref{cc0} and \ref{cc1}. Further, we investigated the graphs $\G_{2,K}$ and 
$\G_{3,K}$ for some finite fields $K$ using computational methods based on 
\cite{Eic07} to solve the isomorphism problem for nilpotent associative 
algebras over finite fields; see Section \ref{cc2}.

Similar to the graphs $\G_{r,K}$ one can also define a directed graph $\G_r$ 
whose vertices correspond one-to-one to the isomorphism types of semigroups of 
coclass $r$. While the graphs $\G_r$ are also of interest, they do not
exhibit the same general features as $\G_{r,K}$. We compare $\G_r$ and
$\G_{r,K}$ briefly in Section \ref{final}.

The idea of using the coclass for the classification of nilpotent algebraic
objects has first been introduced by Leedham-Green \& Newman \cite{LNe80}
for nilpotent groups. We also refer to the book by Leedham-Green \& McKay
\cite{LGM02} for background and many details on the results in the group 
case. Various details of the approach taken here are similar to the 
concepts in group theory. In particular, the idea of searching for periodic 
patterns in coclass graphs as used below also arises in group theory; we 
refer to \cite{DuS01, ELG08} for details. Note though that a nilpotent
semigroup is not a group and hence the coclass theories for groups and
semigroups are independent.

\section{Coclass conjectures for semigroups} 
\label{conj}

In this section we investigate general features of the graph $\G_{r,K}$
for $r \in \N_0$ and arbitrary field $K$.

By construction, every connected component of $\G_{r,K}$ is a rooted tree.
Using basic results on nilpotent semigroups (see~\cite[Lemma 2.1]{Dis11})
one readily shows that $2r$ is an upper bound for the dimension of a root
(that is, the dimension of the corresponding algebra) in $\G_{r,K}$. Thus
$\G_{r,K}$ consists of finitely many rooted trees. We call an infinite path
in a rooted tree {\em maximal} if it starts at the root of the tree.

\begin{conjecture}
\label{conjA}
Let $r \in \N_0$ and $K$ an arbitrary field. 
Then the graph $\G_{r,K}$ has only finitely many maximal infinite paths. 
The number of such paths depends on $r$ but not on $K$.
\end{conjecture}

For an algebra $A$ in $\G_{r,K}$ we denote by $\T(A)$ the subgraph of
$\G_{r,K}$ consisting of all paths that start at $A$. This is a rooted
tree with root $A$. We say that $\T(A)$ is a {\em coclass tree} 
if it contains a unique maximal infinite path. A coclass tree
$\T(A)$ is {\em maximal} if either $A$ is a root in $\G_{r,K}$ or
the parent of $A$ lies on more than one maximal infinite paths.

\begin{remark}
Conjecture \ref{conjA} is equivalent to saying that $\G_{r,K}$ consists 
of finitely many maximal coclass trees and finitely many other vertices.
\end{remark}

We consider the maximal coclass trees in $\G_{r,K}$ in more detail. For a
labelled tree $\T$ we denote with $\nlab{\T}$ the tree without labels.

\begin{conjecture}
\label{conjB}
Let $r \in \N_0$ and $K$ an arbitrary field. 
Let $\T$ be a maximal coclass tree in $\G_{r,K}$ with maximal infinite path
$A_1 \ra A_2 \ra \ldots$ Then $\T$ is {\em weakly virtually periodic}; that 
is, there exist positive integers $l$ and $k$ so that $\nlab{\T}(A_l) \cong 
\nlab{\T}(A_{l+k})$ holds.
\end{conjecture}

The integers $l$ and $k$ with the property of Conjecture \ref{conjB} are
called {\em weak defect} and {\em weak period} of $\nlab{\T}$. Note 
that they are not unique. Every integer larger than $l$ and every multiple 
of $k$ are weak defects and weak periods as well, respectively.

Consider a maximal coclass tree $\T$ of $\G_{r,K}$ with maximal
infinite path $A_1 \ra A_2 \ra \ldots$ Suppose that for some $l$ and
$k$ there exists a graph isomorphism $\mu : \nlab{\T}(A_l) \ra
\nlab{\T}(A_{l+k})$. Then $\mu$ defines a partition of the vertices of
$\T(A_l)$ into finitely many infinite families: for each vertex $B$
contained in $\T(A_l) \setminus \T(A_{l+k})$ define the infinite
family $(B, \mu(B), \mu^2(B), \ldots)$. Hence Conjecture \ref{conjB}
asserts that the unlabelled tree $\nlab{\T}$ can be constructed from a
finite subgraph, provided that a weak defect and a weak period are
known. This implies that $\T$ has finite width. Conjecture \ref{conjA} 
adds that these features of maximal coclass trees extend to all of 
$\G_{r,K}$.
We next exhibit an extension of Conjecture \ref{conjB} incorporating
labels.

\begin{conjecture}
\label{conjC}
Let $r \in \N_0$ and $K$ an arbitrary field. 
Let $\T$ be a maximal coclass tree in $\G_{r,K}$ with maximal infinite path
$A_1 \ra A_2 \ra \ldots$ Then $\T$ is {\em strongly virtually periodic}; 
that is, there exist positive integers $l$ and $k$, a graph isomorphism 
$\mu : \nlab{\T}(A_l) \ra \nlab{\T}(A_{l+k})$ and for every vertex $B$ 
in $\T(A_l) \setminus \T(A_{l+k})$ a rational polynomial $f_B$ so that 
the label of $\mu^i(B)$ equals $f_B(i)$.
\end{conjecture}

The integers $l$ and $k$ with the property of Conjecture \ref{conjC} are
called {\em strong defect} and {\em strong period} of $\T$. As in the
weak case, they are not unique. Further, every strong defect and strong
period are also a weak defect and weak period, but the converse does not
hold in general; compare Section \ref{cc1} for an example.

Conjectures \ref{conjA} and \ref{conjC} suggest the following new approach 
towards a classification up to isomorphism of all nilpotent semigroups of 
fixed coclass $r \in \N_0$. 

\begin{items}
\item[(1)]
Choose an arbitrary field $K$ and classify the maximal infinite paths in 
$\G_{r,K}$.
\item[(2)]
For each maximal infinite path consider its corresponding coclass tree
$\T$ and find a strong defect $l$, a strong period $k$ and an upper bound 
$d$ to the degree of the polynomials of the associated families. 
\item[(3)]
For each maximal coclass tree $\T$ with strong defect $l$, strong period $k$ 
and bound $d$:
\begin{items}
\item[(a)]
Determine the unlabelled tree $\nlab{\T}$ up to depth $l+(d+1)k$. 
\item[(b)]
For each vertex $B$ in the determined part of $\nlab{\T}$ compute its
label.
\end{items}
\item[(4)]
Determine the finite parts of $\G_{r,K}$ outside the maximal coclass trees.
\end{items}

Step (1) is discussed further in Section \ref{infdim} below. For Step (2)
it would be the hope that a constructive proof of the conjectures posed
here might also yield values for strong defect, strong period and bounds 
for the degrees of the arising polynomials.

Steps (3a) and (3b) may be facilitated by two algorithms. The first 
determines up to isomorphism all contracted semigroup algebras $B$ of 
class $c+1$ with $B/B^{c+1} \cong A$ for any given contracted semigroup 
algebra $A$ of class $c$. The second algorithm takes a nilpotent 
associative algebra $A$ of finite dimension and computes up to isomorphism 
all semigroups $S$ with $K S \cong A$. Both algorithms reduce to a finite
computation if the underlying field $K$ is finite. A practical realisation
for the first algorithm in the finite field case may be obtained as 
variation of the method in \cite{Eic07}.

Once the Steps (1) - (4) have been performed, this would allow to construct 
the full graph $\G_{r,K}$ using the graph isomorphism of Conjecture
\ref{conjC}. The polynomials $f_B$ can be interpolated from the given 
information, as there are $d+1$ values $f_B(i)$ available.

\section{The infinite paths in $\G_{r,K}$}
\label{infdim}

In this section we investigate in more detail the infinite paths in 
$\G_{r,K}$ for arbitrary $r \in \N_0$ and arbitrary field $K$. We first 
provide some background for our constructions.

\subsection{Coclass for infinite objects}

Let $O$ be a finitely generated infinite semigroup or a finitely generated 
infinite dimensional associative algebra. Then every quotient $O/O^i$ is 
finitely generated of class at most $i$ and hence is finite (in the semigroup
case) or finite dimensional (in the algebra case). Thus $O/O^i$ has finite
coclass $cc(O/O^i)$. We say that $O$ is {\em residually nilpotent} if 
$\cap_{i \in \N} O^i = 0$ holds. If $O$ is finitely generated and 
residually nilpotent, then we define its {\em coclass} $cc(O)$ as
\[ cc(O) = \lim_{i \ra \infty} cc(O/O^i).\]

The coclass of $O$ can be finite or infinite. It is finite if and only if
there exists $i \in \N$ so that $|O^{j+1} \setminus O^j| = 1$ (in the
semigroup case) or $\dim(O^j/O^{j+1}) = 1$ (in the algebra case) for all 
$j \geq i$. If we say that $O$ has `finite coclass', then this implies
that $O$ is finitely generated and residually nilpotent.

\subsection{Inverse limits of algebras and semigroups}

Consider a maximal infinite path $A_1 \ra A_2 \ra \ldots$ in $\G_{r,K}$ and
let $\hat{A} = \prod_{i \in \N} A_i$ be the Cartesian product of the algebras
on the path. If $A_1$ has class $c$, then $A_j$ has class $j+c-1$ and thus
$A_{j+1}/A_{j+1}^{j+c} \cong A_j$ for every $j \in \N$. For every $j \in \N$ 
we choose an epimorphism $\nu_j : A_{j+1} \ra A_j$ with kernel 
$A_{j+1}^{j+c}$. We define the {\em inverse limit} of the algebras on the path as
\[ A = \left\{ (a_1, a_2, \ldots) \in \hat{A} \mid \nu_j(a_{j+1}) = a_j 
     \mbox{ for every } j \in \N \right\}.\]

The inverse limit $A$ is an infinite dimensional associative $K$-algebra
which satisfies $A / A^{j+c} \cong A_j$ for every $j \in \N$. Thus $A/A^2$
is finite dimensional and hence $A$ is finitely generated. It is also 
residually finite and has coclass $r$. Further, each algebra on the maximal 
infinite path can be obtained as quotient of $A$ and thus $A$ fully 
describes the considered maximal infinite path. We summarize this as follows.

\begin{theorem}
\label{infpath}
Let $r \in \N_0$ and $K$ an arbitrary field.
For every maximal infinite path in $\G_{r,K}$ there exists an infinite
dimensional associative $K$-algebra of coclass $r$ which describes 
the path. 
\end{theorem}

Isomorphic algebras of the type considered in Theorem \ref{infpath} 
describe the same infinite path. Hence an approach to the classification 
of the maximal infinite paths in $\G_{r,K}$ is the determination up to 
isomorphism of the infinite dimensional associative $K$-algebras $A$ of 
coclass $r$ whose quotients $A/A^j$ are contracted semigroup algebras
for every $j \in \N$. Conjecture \ref{conjA} is equivalent to saying that
there are only finitely many of these objects up to isomorphism. The
following theorem describes these algebras in more detail.

\begin{theorem}
Let $r \in \N_0$ and $K$ an arbitrary field.
Each infinite dimensional associative $K$-algebra of coclass $r$ which
describes an infinite path in $\G_{r,K}$ is isomorphic to a contracted 
semigroup algebra $KS$ for an infinite semigroup $S$ of coclass $r$.
\end{theorem}

\begin{proof}
Let $A$ be an infinite dimensional associative $K$-algebra of coclass $r$
which describes an infinite path in $\G_{r,K}$. Then there exists an $i \in 
\N$ so that $A/A^j$ is a contracted semigroup algebra of coclass $r$ for 
every $j \geq i$. Each of the quotients $A/A^j$ may be the contracted 
semigroup algebra for several non-isomorphic semigroups. Our aim is to show 
that for every $j \geq i$ there exists a semigroup $S_j$ so that $A/A^j 
\cong K S_j$ and $S_j \cong S/S^j$ for an infinite semigroup $S$ of 
coclass $r$.

We define a graph $\L$ whose vertices correspond one-to-one to the isomorphism
types of semigroups whose contracted semigroup algebra is isomorphic to a
quotient $A/A^j$ for some $j \geq i$. We connect two semigroups in $\L$
by a directed edge $U \ra T$ if $T/T^c \cong U$, where $c$ is the class of
$T$. If a semigroup $T$ satisfies $KT \cong A/A^j$ for some $j > i$, then 
$T$ has class $j-1$ and $U \cong T/T^{j-1}$ satisfies $KU \cong A/A^{j-1}$.
Hence each connected component of $\L$ is a tree with a root of class $i-1$
and coclass $r$. There is at least one infinite connected component $\M$ of
$\L$. By K\"onig's Lemma, the tree $\M$ contains an infinite path, say
$M_i \ra M_{i+1} \ra \ldots$ Let $S$ be the inverse limit of the semigroups
on this infinite path. Then $S$ is an infinite semigroup with $S/S^j
\cong M_j$ and $KM_j \cong A/A^j$ for every $j \geq i$. In particular, the
semigroup $S$ has finite coclass $r$. 

It remains to show that $S$ satisfies $KS \cong A$. This follows from the
construction of $S$, as the following diagram is commutative, where upwards
arrows denote embeddings of semigroups in their contracted semigroup 
algebras:
\[ \begin{array}{ccccccc}
   A/A^i & \ra & A/A^{i+1} & \ra & A/A^{i+2} & \ra & \ldots \\
   \uparrow & & \uparrow & & \uparrow && \\
   S/S^i & \ra & S/S^{i+1} & \ra & S/S^{i+2} & \ra & \ldots \\
   \end{array}. \]
This completes the proof.
\end{proof}

\subsection{Examples}

Consider the polynomial algebra in one indeterminate and let $I_K$ denote
its ideal consisting of all polynomials with zero constant term. Then $I_K$
is an explicit construction for the free non-unital associative algebra
on one generator over the field $K$. It is isomorphic to the contracted
semigroup algebra $KS$ with $S \cong (\N_0, +)$ and it has coclass $0$. 
Hence it describes an infinite path in $\G_{0,K}$. In Section \ref{cc0} 
we observe that it describes the unique maximal infinite path in $\G_{0,K}$.

Examples of infinite dimensional contracted semigroup algebras of higher 
coclass can be obtained inductively using the following process. Let $S$
be an infinite semigroup of coclass $r-1$. An {\em annihilator extension} 
of $S$ is an infinite semigroup $T$ so that $T$ contains a non-zero 
element $t \in \Ann(T)$ with $S \cong T / \langle t \rangle$. 

\begin{lemma}
Let $r \in \N$ and let $S$ be an infinite semigroup of coclass $r-1$. 
Each annihilator extension $T$ of $S$ is an infinite semigroup of coclass 
$r$.
\end{lemma}

\begin{proof}
Consider the sequence $T \geq T^2 \geq T^3 \geq \ldots$ and define $c \in \N$
via $t \in T^c \setminus T^{c+1}$. Let $\nu : T \ra S$ be an epimorphism with 
kernel $\langle t \rangle$. As $\nu(T^i) = S^i$, we obtain that $T/T^i \cong 
S/S^i$ for $1 \leq i \leq c$ and for $i \geq c+1$ we obtain that $|T/T^i| = 
|S/S^i|+1$. Thus $T$ is finitely generated and residually nilpotent and it 
satisfies $cc(T/T^i) = cc(S/S^i)+1$ for $i \geq c+1$. Thus $T$ is an infinite
semigroup of coclass $cc(T) = cc(S)+1$.
\end{proof}

If $A$ is an infinite dimensional contracted semigroup algebra of coclass 
$r-1$, then $A = KS$ for an infinite semigroup $S$ of coclass $r-1$. Thus 
every annihilator extension $T$ of $S$ gives rise to an infinite dimensional
contracted semigroup algebra of coclass $r$. 

We exhibit an explicit example for this process. For this purpose let $L_K$ 
denote the 1-dimensional nilpotent algebra of class $1$. Then $L_K$ is 
isomorphic to the contracted semigroup algebra $KZ_2$, where $Z_n$ is the
zero semigroup with $n$ elements. For every $r\in\N$ the algebra 
\begin{equation}
\label{eq_MKr}
M_{K,r} = I_K \oplus \bigoplus_{i=1}^r L_K
\end{equation}
is an infinite dimensional contracted semigroup algebra of coclass
$r$. As underlying semigroup one can choose the zero union of $(\N,+)$
and $Z_{r+1}$, that is the semigroup on $\N \cup Z_{r+1}$ in which
mixed products equal $0\in Z_{r+1}$. For $r>0$ this is an annihilator
extension of the zero union of  $(\N,+)$ and $Z_{r}$ corresponding to
$M_{K,r-1}$.

We close this section by posing the following question.

\begin{question}
\label{conjE}
Does every infinite dimensional algebra which describes an infinite
path in $\G_{r,K}$ arise as contracted semigroup algebra for a semigroup
which is an annihilator extension?
\end{question}

\section{The minimal generator number}
\label{sec_mingen}

A nilpotent semigroup $S$ has a unique minimal generating set $S \setminus
S^2$. Its cardinality corresponds to the dimension of the quotient $KS / 
(KS)^2$ and thus to the minimal generator number of the algebra $KS$. Hence
$KS \cong KT$ implies that the nilpotent semigroups $S$ and $T$ have the
same minimal generator number. Further, if two algebras in $\G_{r,K}$ are 
connected, then they have the same minimal generator number. This allows 
to define the subgraph $\G_{r,K,d}$ of $\G_{r,K}$ corresponding to the 
nilpotent semigroups of coclass $r$ with minimal generator number $d$.

A nilpotent semigroup of coclass $r$ has at most $r+1$ generators. 
Thus $\G_{r,K,d}$ is empty for $d \geq r+2$ (and also for $d=1$ if $r > 0$). 
The extremal case $\G_{r,K,r+1}$ can be described in more detail as the 
following theorem shows. Recall that $Z_n$ is the zero semigroup with $n$ 
elements and $M_{K,r}$ is defined in~\eqref{eq_MKr}.

\begin{theorem}
\label{mingen}
Let $r \in \N_0$ and $K$ an arbitrary field. Then 
$\G_{r,K,r+1}$ consists of a unique maximal coclass tree with corresponding 
infinite dimensional algebra $M_{K,r}$. The root of the maximal coclass tree
is $K Z_{r+2}$ if $r > 0$ and $K Z_{1}$ if $r = 0$.
\end{theorem}

\begin{proof}
The semigroup $Z_{r+2}$ has $r+2$ elements, minimal generator number $r+1$, 
class $1$ and thus coclass $r$. If $r > 0$, then $Z_{r+2}$ is the unique 
semigroup of coclass $r$ and order at most $r+2$ and hence $KZ_{r+2}$ is a 
root in $\G_{r,K,r+1}$. If $r = 0$, then $Z_1$ is a root of $\G_{0,K,1}$.

In the following we assume that $r > 0$. The case $r=0$ is similar and we
leave it to the reader. Let $S$ be an arbitrary semigroup of class $c$ such 
that $KS$ is in $\G_{r,K,r+1}$. We show by induction on $|S|$ that there 
exists a path from $KZ_{r+2}$ to $KS$. As $Z_{r+2}$ is the only semigroup 
of coclass $r$, order at most $r+2$ and minimal generator number $r+1$, we 
may assume that $|S| > r+2$. From $|S \setminus S^2|=r+1$ follows $|S^2|=c$ 
and hence $|S^c|=2$. Thus $S/S^c$ is a semigroup of coclass $r$ with minimal 
generator number $r+1$ and with $|S| - 1$ elements. This yields that there 
is an edge from $KS/(KS)^c \cong K (S/S^c)$ to $KS$. By induction, there 
exists a path from $KZ_{r+2}$ to $K (S/S^c)$ and hence to $KS$. This proves 
that  $\G_{r,K,r+1}$ is connected.

The infinite dimensional algebra $M_{K,r}$ has coclass $r$ and minimal
generator number $r+1$ and it is a contracted semigroup algebra. It defines 
a maximal infinite path in $\G_{r,K,r+1}$. It remains to show that this 
maximal infinite path is unique. Let $A$ be an arbitrary infinite dimensional 
assocative algebra of coclass $r$ with $r+1$ generators. Then $\dim(A/A^2) 
= r+1$ and $\dim(A^i/A^{i+1}) = 1$ for every $i \geq 2$. Let $v,w,x\in A$ 
such that $vwxA^4$ generates $A^3/A^4$. Then both $vwA^3$ and $wxA^3$ 
generate $A^2/A^3$ and hence $vw = kwx$ for some $k\in K$. This yields 
$vwx=kwxx$ and hence $x^2A^3$ is a generator of $A^2/A^3$. By induction, 
it follows that $x^iA^{i+1}$ is a generator of $A^i/A^{i+1}$ for every 
$i\geq 2$. Now choose elements $x_1, \ldots, x_r \in A$ that together with 
$x$ correspond to a basis of $A/A^2$. Then these elements generate $A$. 
A basis of $A^2$ has the form $\{ x^j \mid j \geq 2\}$. Thus for $i \in 
\{1, \ldots, r\}$ we find that 
\[ x x_i = \sum_{j=2}^\infty k_{ij} x^j \in A^2. \]
We replace $x_i$ by 
\[ y_i = x_i - \sum_{j=2}^\infty k_{ij} x^{j-1}\]
and thus obtain a new minimal generating set $x, y_1, \ldots, y_r$ of $A$ which
satisfies $x y_i = 0$ by construction. For $i,j \in \{1,\ldots, r\}$
and consider the product $y_i y_j$. Then $y_i y_j = \sum_{h=2}^\infty k_h x^h 
\in A^2$. As $x y_i = 0$, it follows that $x y_i y_j = 0$ and thus 
$\sum_{h=2}^\infty k_h x^{h+1} = 0$. This implies that all coefficients 
$k_h$ equal $0$ and hence $y_i y_j = 0$ for every $i, j \in \{1, \ldots, r\}$. 
This yields that $A \cong M_{K,r}$.
\end{proof}

\section{The graph $\G_{0,K}$}
\label{cc0}

The semigroups of coclass $0$ are well-known; for every order $n\in\N$ there
exists exactly one such semigroup with presentation $\langle u\mid 
u^n=u^{n+1}\rangle$. Together with the result from
Theorem~\ref{mingen} this leads to the following theorem.

\begin{theorem}
Let $K$ be an arbitrary field. The graph $\G_{0,K}$ consists of a
unique maximal coclass tree with root $K \!Z_1$. This tree is strongly
virtually periodic with strong defect $1$, strong period $1$, and the
single associated polynomial $f_{K\!Z_1}(x) = 1$.
\end{theorem}

\section{The graph $\G_{1,K}$}
\label{cc1}

We determine the graph $\G_{1,K}$ for arbitrary fields $K$ using the 
classification \cite{Dis10, Dis11} of nilpotent semigroups of coclass $1$.
As preliminary step, note that a nilpotent semigroup of coclass $1$ has 
at least $3$ elements. Up to isomorphism there exist exactly one
semigroup of coclass  $1$ with $3$ elements, namely $Z_3$, and nine
semigroups with $4$ elements.
\begin{theorem}
Let $K$ be an arbitrary field. 
\begin{items}
\item[\rm (1)]
The graph $\G_{1,K}$ consists of a unique maximal coclass tree $\T$
with root $K \!Z_{3}$ and corresponding infinite dimensional algebra
$M_{K,1}$ (defined in \eqref{eq_MKr}).
\item[\rm (2)]
The tree $\T$ is strongly virtually periodic with strong defect $2$ and 
strong period $2$. Let $A_1 \ra A_2 \ra \ldots$ denote the maximal infinite 
path of $\T$. For each algebra $B\in \T(A_2)\setminus\T(A_4)$
the polynomial corresponding to $B$ has degree at most $1$.
\begin{items}
\item[\rm (a)]
If $\sqrt{-1} \in K$, then $\T(A_2)\setminus\T(A_4)$ consist of $A_2$,
$4$ algebras with $A_2$ as parent, and $3$ algebras with $A_3$ as
parent; see the right box of Figure~\ref{figcc1}. 
\item[\rm (b)]
If $\sqrt{-1} \in K$, then $\T(A_2)\setminus\T(A_4)$ consist of $A_2$,
$4$ algebras with $A_2$ as parent, and $4$ algebras with $A_3$ as
parent; see the left box of Figure~\ref{figcc1}. 
\end{items}
\end{items}
\end{theorem}

\begin{figure}[thb]
\begin{center}
\includegraphics[scale=0.5]{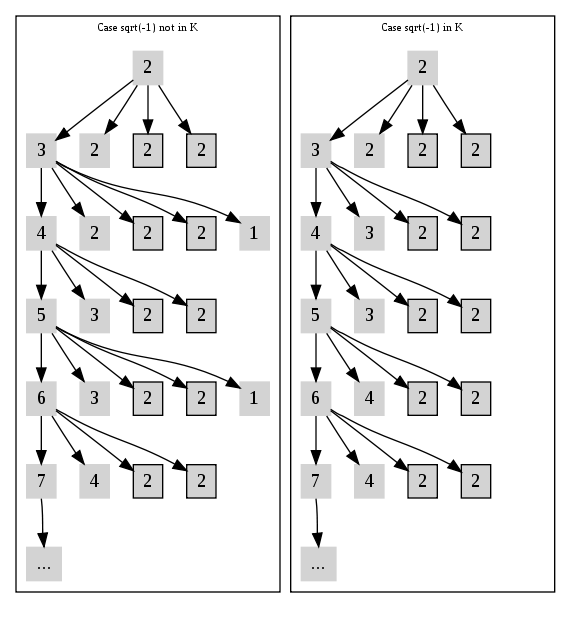}
\end{center}
\vspace{-1cm}
\caption{Description of $\T(A_2)$ in $\G_{1,K}$ with root of 
dimension $3$. Vertices with box correspond to non-commutative algebras.
The polynomials of degree $1$ are $2x+2$ and $2x+3$ for the two families 
on the infinite path and $x+2$ for the other two families.}
\label{figcc1}
\end{figure}

\begin{proof}
The first part of the statement is true by Theorem \ref{mingen}. To prove 
the second part we use the classification from \cite{Dis11}: there are the 
following $n+2+ \lfloor n/2 \rfloor$  isomorphism types of semigroups of 
order $n$ and coclass $1$ for $n \geq 5$:
\begin{items}
\bulit 
$H_k = \langle u,v \mid 
u^{n-1}=u^n, uv=u^k, vu=u^k,v^2=u^{2k-2} \rangle, 2\leq k\leq n-1$;
\bulit
$J_k=\langle u,v \mid 
u^{n-1}=u^n, uv=u^k, vu=u^k,v^2=u^{n-2} \rangle, n/2 < k\leq n-1$;
\bulit
$X = \langle u,v \mid u^{n-1}=u^n, uv=u^{n/2}, vu=u^{n/2},v^2=u^{n-1}
  \rangle, \mbox{ if }n \equiv 0 \mod 2$;
\bulit
$N_1=\langle u,v \mid u^{n-1}=u^n, uv=u^{n-1}, vu=u^{n-2},v^2=u^{n-2} \rangle$;
\bulit
$N_2=\langle u,v \mid u^{n-1}=u^n, uv=u^{n-2}, vu=u^{n-1},v^2=u^{n-2} \rangle$;
\bulit
$N_3=\langle u,v \mid u^{n-1}=u^n, uv=u^{n-1}, vu=u^{n-2},v^2=u^{n-1} \rangle$;
\bulit
$N_4=\langle u,v \mid u^{n-1}=u^n, uv=u^{n-2}, vu=u^{n-1},v^2=u^{n-1} \rangle$.
\end{items}

We now show which of these semigroups give rise to isomorphic algebras.
\medskip 

Show that $K H_2 \cong K H_k$ for $3 \leq k \leq n-1$ holds.\\
Define $\mu : K H_2 \ra K H_k$ via $\mu(u) = u+u^{k-1}$ and $\mu(v) = u+v$.
As $(u+u^{k-1})^m = \sum_{i=0}^m {m \choose i} u^{m-i} (u^{k-1})^i = 
\sum_{i=0}^m {m \choose i} u^{m+(k-2)i}$ for $1\leq m\leq n-1$, it follows
that the elements $u+u^{k-1}$ and $u+v$ generate $K H_k$. The images of 
$u$ and $v$ under $\mu$ satisfy the relations of $H_2$ and hence $\mu$ 
induces an epimorphism. As $K H_k$ and $K H_2$ have the same dimension, 
it follows that $\mu$ is an isomorphism.
\medskip

Show that $K J_{n-1} \cong K J_k$ for $n/2 < k < n-1$ and 
$K X \cong K J_{n-1}$ if $n$ is even and $\sqrt{-1} \in K$ hold.\\
For the first part, define $\mu : K J_{n-1} \ra K J_k$ via $\mu(u) = u$ 
and $\mu(v) = v-u^{k-1}$. For the second part, define $\mu : K X \ra 
K J_{n-1}$ via $\mu(u) = u$ and $\mu(v) = u^{n/2-1} - \sqrt{-1} v$. Then 
as above, it follows that $\mu$ extends to an isomorphism. 
\medskip

Show that $K N_1 \cong K N_2$ and $K N_3 \cong K N_4$ hold.\\
Note that $(N_1,N_2)$ and $(N_3,N_4)$ are pairs of 
anti-isomorphic semigroups. For each $i \in \{1, \ldots, 4\}$, the 
subsemigroup $\langle u, u^{n-3}-v\rangle$ yields a basis of $K N_i$ 
and is isomorphic to the dual semigroup of $N_i$. Hence $K N_1 \cong 
K N_2$ and $K N_3 \cong K N_4$ follow.
\medskip

It remains to show that we have determined all isomorphisms. First, 
we consider $K H_{n-1}$ and $K J_{n-1}$. These are both 
commutative algebras; the first has an annihilator of dimension 2
generated by $v$ and $u^{n-2}$ and the second has an annihilator of
dimension 1 generated by $u^{n-2}$. Hence the algebras are non-isomorphic.
Secondly, we consider  $K N_1$ and $K N_3$. These are both
non-commutative algebras and they both have an annihilator of dimension 1; 
the first has a right annihilator of dimension 2 generated by $v$ 
and $u^{n-2}$ and the second has a right annihilator of dimension 1
generated by $u^{n-2}$. Hence the algebras are non-isomorphic. This
proves our claim in the case $\sqrt{-1} \in K$ or $n$ odd. In the case 
$\sqrt{-1} \not \in K$ and $n$ even, there is the additional algebra 
$K X$. This is a commutative algebra whose annihilator has dimension 1;
hence we have to distinguish $K X$ from $K J_{n-1}$. Assume that $\mu : 
K X \ra K J_{n-1}$ is an isomorphism and denote $\mu(v) = av + 
\sum_{i=1}^{n-2} b_i u^i \in K J_{n-1}$. Then $\mu(v)^2 = a^2 v^2+
(\sum_{i=1}^{n-2} b_i u^i)^2$ in $K J_{n-1}$, as $uv = vu = 0$ holds. 
Note that $v^2 = u^{n-2}$ in $K J_{n-1}$ and $\mu(v)^2 = \mu(v^2) = 
0\in K J_{n-1}$ as $v^2=0$ in $KX$. 
An inspection of the coefficients now shows that $b_i = 0$ for $1 \leq i 
\leq n/2-2$. The coefficient of $u^{n-2}$ in $\mu(v)^2$ thus is 
$a^2+b_{n/2-1}^2$. Since $\sqrt{-1} \not \in K$, it follows that 
$a = b_{n/2-1} = 0$. This yields that $\mu(v) \in \langle u^{n/2}, 
u^{n/2+1}, \ldots, u_{n-2} \rangle \leq (K J_{n-1})^2$. Hence $\mu$ 
is not surjective, a contradiction.
\medskip

We determine the edges of $\G_{1,K}$. Consider a semigroup $S$ of order
$n$ from the above classification. In the quotient $S/S^{n-2}$ the two
elements $u^{n-2}$ and $u^{n-1}$ are identified, and hence the quotient 
is isomorphic to a semigroup of type $H_k$ of order $n-1$. Note that the 
later is valid for $n=5$ also, as the semigroups $H_k$ can be defined for 
order $4$ as well.
\medskip

The labels of the vertices in $\G_{1,K}$ follow immediately from the
classification. This implies that $\G_{1,K}$ has strong defect 2 and
strong period 2. (In fact, both values are minimal.)
\end{proof}

Images of the parts of $\G_{1,K,2}$ corresponding to semigroups
of order at most $12$ for $K = GF(p)$ with $p \leq 23$ can be found
at~\cite{homepage}.
 
\section{Computational experiments for $\G_{2,K}$}
\label{cc2}

A classification of semigroups of coclass $2$ is available in
\cite{Dis10,Dis11}. We
used it to investigate $\G_{2,K}$ computationally, applying the
isomorphism test for associative nilpotent algebras over finite fields
in~\cite{Eic07}. Semigroups of coclass $2$ have a minimal
generating set of size $2$ or $3$. We know from Section~\ref{sec_mingen}
that these two cases lead to independent subgraphs $\G_{2,K,2}$ and
$\G_{2,K,3}$ of $\G_{2,K}$ which shall be considered separately.

We have determined the part of $\G_{2,K,2}$ corresponding to semigroups
of order at most $12$ for $K = GF(p)$ with $p \leq 23$. Images of the
graphs are available at~\cite{homepage}. We conjecture
that for every field $K$ the graph $\G_{2,K,2}$ has five maximal infinite
paths which are described by the following infinite dimensional
algebras:
\begin{items}
\bulit
$\langle a, b \mid b^2 = ba = a^2b = 0 \rangle$
with annihilator $\langle ab \rangle$;
\bulit
$\langle a, b \mid b^2 = ab = ba^2 = 0 \rangle$
with annihilator $\langle ba \rangle$;
\bulit
$\langle a, b \mid b^3 = ab = ba = 0 \rangle$
with annihilator $\langle b^2 \rangle$;
\bulit
$\langle a, b \mid b^2 = aba = 0, ab = ba \rangle$
with annihilator $\langle ba \rangle$;
\bulit
$\langle a, b \mid b^2 = ba, ab = b^2a = 0 \rangle$
with annihilator $\langle ba \rangle$.
\end{items}
Using these presentations to define semigroups with zero we obtain
infinite semigroups that are annihilator extensions of the semigroup
underlying $M_{K,1}$ and whose contracted semigroup algebras are the
algebras defined by the presentations. 
If the conjecture on the number of infinite paths holds, then
$\G_{2,K,2}$ contains five maximal coclass
trees. Figure~\ref{figcc2d2} exhibits the respective trees of the
computed graph for $K=GF(5)$.
Furthermore our computational evidence suggests the following:

\begin{items}
\item[$\bullet$]
the graph $\G_{2,GF(p),2}$ depends on $p \bmod 4$ only;
\item[$\bullet$]
the vertices in $\G_{2,GF(p),2}$ outside a maximal coclass tree have 
dimension $4$ or $5$;
\item[$\bullet$]
the roots of the maximal coclass trees of $\G_{2,GF(p),2}$ have dimension $4$;
\item[$\bullet$]
each maximal coclass tree in $\G_{2,GF(p),2}$ is strongly virtually
periodic; one tree has strong defect $1$ and strong period $1$, the
other four trees have strong defect $2$ and strong period $2$;
\item[$\bullet$]
the arising strong defects and strong periods are independent of the field.
\end{items}

Additionally the polynomials describing the labels in the
periodic parts of the maximal coclass trees have degree at most $1$, a
fact that follows from the classification in~\cite{Dis10,Dis11}.

\begin{figure}[thb]
\begin{center}
\includegraphics[scale=0.4]{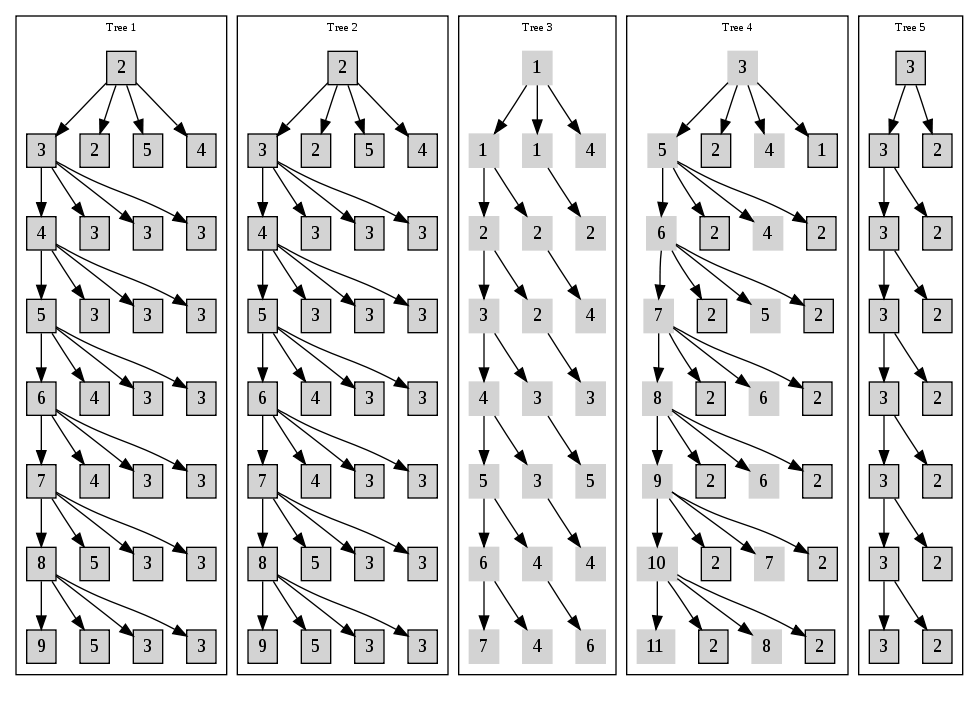}
\end{center}
\vspace{-1cm}
\caption{Maximal coclass trees in $\G_{2,GF(5),2}$ up to depth 12.}
\label{figcc2d2}
\end{figure}

The graph $\G_{2,K,3}$ is known to consist of a single maximal coclass
tree with root $KZ_4$ and infinite paths corresponding to $M_{K,2}$ by
Theorem~\ref{mingen}. We have determined the part of $\G_{2,K,3}$
corresponding to semigroups of order at most $12$ for $K = GF(p)$ with
$p \leq 5$. Images of the graphs are available at~\cite{homepage}.
In all three cases we found the graph to appear strongly
virtually periodic with strong defect $2$, and strong period $2$. In
accordance with the results from~\cite{Dis10,Dis11} the labels can be
described by quadratic polynomials.

\section{Computational experiments for $\G_{3,K}$}
\label{cc3}

For the semigroups of coclass $3$ there is no general classification known. 
We computed the semigroups of coclass $3$ and order at most $17$ up to 
isomorphism using the code provided in~\cite[Appendix C]{Dis10}. Then we
determined the part of $\G_{3,K,2}$ correspondig to these semigroups for 
$K = GF(p)$ with $p \leq 23$. Images of the graphs are available 
at~\cite{homepage}. We summarise our observations:
\begin{items}
\bulit
the graph $\G_{3,GF(p),2}$ depends on $p \bmod 4$ only;
\bulit
the graph $\G_{3,GF(p),2}$ has $15$ maximal infinite paths of which
$4$ correspond to commutative algebras; 
\bulit
the vertices in $\G_{3,GF(p),2}$ outside a maximal coclass tree have 
dimension $6$, $7$ or $8$;
\bulit
the roots of the maximal coclass trees of $\G_{3,GF(p),2}$ have
dimension $5$, $6$ or $7$;
\bulit
the maximal coclass trees in $\G_{3,GF(p),2}$ are strongly virtually
periodic with strong defect at most $3$ and strong period at most $6$;
\bulit
the arising strong periods are independent of the field;
\bulit
the polynomials describing the labels have degree at most $1$.
\end{items}

These observations for $\G_{3,GF(p),2}$ are of particular
interest as this is the first case in which some of the semigroups
contain products of three elements that lie in different monogenic
subsemigroups. In fact, $\G_{3,GF(p),2}$ has in many aspects more 
complex features than all other considered graphs.

We have investigated the part of $\G_{3,K,3}$ correspondig to semigroups of
order at most $12$ for $K = GF(2)$ only. There appear to be $21$
maximal infinite paths in $\G_{3,GF(2),3}$ with $5$ paths
corresponding to commutative algebras. 

The graph $\G_{3,K,4}$ has $1$ maximal infinite paths corresponding
to the commutative algebra $M_{K,3}$ by Theorem~\ref{mingen}.

\section{Concluding comments}
\label{final}

Similar to the graphs $\G_{r,K}$ one can define a graph $\G_r$ whose 
vertices correspond one-to-one to the isomorphism types of semigroups 
of coclass $r$. Two vertices are adjoined by a directed edge $T \ra S$
if $S/S^c \cong T$ where $c$ is the class of $S$.
It follows directly from \cite[Lemma 3.1]{Dis11} that the graph $\G_r$
does not have finite width (unless $r=0$).

The additional use of the contracted semigroup algebras in the definition
of $\G_{r,K}$ induces a dependence on the underlying field $K$, but it has 
the significant advantage that the graphs $\G_{r,K}$ seem to have finite 
width and exhibit periodic patterns which can be described in a compact way. 
Further, the considered field $K$ seems to have no influence on the important
aspects of the periodicity. 

\def\cprime{$'$}


\begin{thebibliography}{1}

\bibitem{Dis10}
A.~Distler.
\newblock {\em Classification and Enumeration of Finite Semigroups}.
\newblock Shaker Verlag, Aachen, 2010.
\newblock {\em also} PhD thesis, University of St Andrews, 2010,
  \url{http://hdl.handle.net/10023/945}.

\bibitem{Dis11}
A.~Distler.
\newblock Finite nilpotent semigroups of small coclass.
\newblock Preprint, 2012.
\newblock See \url{http://arxiv.org/abs/1205.2817}.

\bibitem{homepage}
A.~Distler and B.~Eick.
\newblock Coclass graphs for semigroup of small coclass.
\newblock See \url{http://www.icm.tu-bs.de/~beick/grph/index.html}.

\bibitem{DuS01}
M.~du~Sautoy.
\newblock Counting $p$-groups and nilpotent groups.
\newblock {\em Inst. Hautes Etudes Sci. Publ. Math.}, 92:63 -- 112, 2001.

\bibitem{Eic07}
B.~Eick.
\newblock Computing automorphism groups and testing isomorphisms for modular
  group algebras.
\newblock {\em J. Algebra}, 320(11):3895--3910, 2008.

\bibitem{ELG08}
B.~Eick and C.~Leedham-Green.
\newblock On the classification of prime-power groups by coclass.
\newblock {\em Bull. Lond. Math. Soc.}, 40(2):274--288, 2008.

\bibitem{LGM02}
C.~R. Leedham-Green and S.~McKay.
\newblock {\em The structure of groups of prime power order}.
\newblock London Mathematical Society Monographs. Oxford Science Publications,
  2002.

\bibitem{LNe80}
C.~R. Leedham-Green and M.~F. Newman.
\newblock Space groups and groups of prime-power order {I}.
\newblock {\em Archiv der Mathematik}, 35:193 -- 203, 1980.

\end{thebibliography}
\end{document}